\documentclass[12pt]{article}

\usepackage{amssymb, amsmath, amsthm}
\usepackage{graphicx}
\usepackage{subcaption}  
\usepackage{hyperref}
\hypersetup{
    colorlinks=true,
    linkcolor=blue,
    filecolor=blue,      
    urlcolor=cyan,
    citecolor = blue
}

\usepackage[usenames,dvipsnames]{xcolor}
\usepackage{todonotes}
\usepackage{xparse}
\usepackage{comment}

\definecolor{C0}{HTML}{1f77b4}
\definecolor{C1}{HTML}{ff7f0e}
\definecolor{C2}{HTML}{2ca02c}
\definecolor{C3}{HTML}{d62728}

\newcommand{\PP}{\mathbb{P}}
\newcommand{\Prob}[1]{\mathbb{P}\left( #1 \right)}
\newcommand{\E}[1]{\mathbb{E}\left[ #1 \right]}

\newtheorem{theorem}{Theorem}
\newtheorem{lemma}{Lemma}

\theoremstyle{remark}

\usepackage[ruled, vlined]{algorithm2e}
\SetKwInput{KwInput}{Input}
\SetKwInput{KwOutput}{Output}
\SetKwInput{CNStep}{Common neighbours check}
\SetKwInput{CliqueStep}{Clique check}
\SetKwInput{KwStart}{Initialization step}
\SetKwInput{KwMinStep}{Minimization step}
\SetKwInput{KwEnd}{Final step}
\SetKwInput{KwInitializationStep}{Initialize}
\SetKwInput{KwUpdateStep}{Update}
\SetKwInput{KwDynUpdateStep}{Dynamic update}
\SetKwInput{KwProcessStep}{Process}
\SetKwInput{KwReturn}{Return}
\SetKwInput{KwGlobalStep}{Global step}
\SetKwInput{KwLocalStep}{Local improvement}

\title{Planted clique recovery in \\ random geometric graphs}
\author{Konstantin Avrachenkov, Andrei Bobu, \\ Nelly Litvak, Riccardo Michielan}

\begin{document}

\maketitle


\abstract{We investigate the problem of identifying planted cliques in random geometric graphs, focusing on two distinct algorithmic approaches: the first based on vertex degrees (VD) and the other on common neighbours (CN). We analyse the performance of these methods under varying regimes of key parameters, namely the average degree of the graph and the size of the planted clique. We demonstrate that exact recovery is achieved with high probability as the graph size increases, in a specific set of parameters. Notably, our results reveal that the CN-algorithm significantly outperforms the VD-algorithm. In particular, in the connectivity regime, tiny planted cliques (even edges) are correctly identified by the CN-algorithm, yielding a significant impact on anomaly detection. Finally, our results are confirmed by a series of numerical experiments, showing that the devised algorithms are effective in practice.}

\maketitle

\section{Introduction}

Graphs give a natural way to model real-world networks. One of the most natural questions that arises in such modelling is the identification of nodes or groups of nodes with certain properties. For instance, many researchers are interested in finding large cliques in graphs that correspond to tightly connected communities in a real network. One possible variant of such a task is to reveal `suspicious' structures or anomalies in the graph that cannot naturally appear in the corresponding network.

The problem of identification of artificial large cliques in Erd\H{o}s--R\'enyi graphs is a classical problem in graph theory. It is called the \textit{planted} (or \textit{hidden}) \textit{clique problem} and can be formalised as follows. Consider a realization of the Erd\H{o}s--R\'enyi graph $G^{ER}(n, 1/2)$ with $n$ nodes and the probability of an edge 1/2. Choose randomly its $k$ nodes. These nodes are declared a \textit{planted clique}, and we add all missing edges between them. The task is to identify the planted clique with high probability using a polynomial-time algorithm. 

The planted clique problem appeared in 1995 with the work of Ku{\v{c}}era \cite{kuvcera1995expected}. He found that if $k  = \Omega(\sqrt{n \log n})$, then the problem can be solved by a simple algorithm based on comparing the degrees of vertices. Let us notice that the largest clique in an Erd\H{o}s--R\'enyi graph $G^{ER}(n, 1/2)$ is of size $\log_2 n (1+o(1))$ a.\,s. Therefore, such an algorithm may be far from optimal. In 1998, it was proven in \cite{alon1998finding} that for any $c > 0$, there is a polynomial-time algorithm that recovers a planted clique with high probability for $k > c\sqrt {n} $. Certain classes of polynomial algorithms have been proved to fail for $k = o(\sqrt n)$, but a general barrier has not yet been found (see discussion in \cite{gamarnik2024landscape}--\cite{hazan2011hard}).

It is surprising and curious that such an interesting and natural problem has not been stated for other models of random graphs. In this paper, we pose a similar question for random geometric graphs, which have become a fairly popular research object lately. Random geometric graphs resemble real social, technological, and biological networks in many aspects due to their tendency to cluster \cite{higham2008fitting}--\cite{preciado2009spectral}. Additionally, they can be useful in statistics and machine learning tasks: correlations between observations in a dataset can be represented by links between corresponding vertices (see, e.g., \cite{ariascastro2015detecting}). Artificial cliques can indicate a malicious intervention in the network or issues with the data. 

Another curious example of a possible application is steganography (see \cite{ker2017square} and \cite{ker2008square}). The idea is that a hidden message (\textit{stego}) is put inside a big piece of information (\textit{cover}, which can be, say, an image). The adjacency matrices can represent such pieces, while a planted clique placed within them corresponds to a stego (the Ising model from \cite[Ex. 2.1]{guyon1995random} appears to be similar to that setting). One of the problems of steganography is to determine how large a stego can be that cannot be detected by any algorithm. This is perfectly consistent with the planted clique recovery problem. In steganography, it is usually assumed that the pixels for the stego are chosen independently, whereas in the case of the planted clique problem, they are dependent (as elements of the adjacency matrix). Nevertheless, we think that a formulation with dependencies might be even more realistic and interesting for steganography. 

One of the key properties of random geometric graphs is the dependency between edges that results from their geometrical structure. In these graphs, the nodes have corresponding spatial positions and the probability of edge appearance depends on the distance between these positions (in contrast to Erd\H{o}s--R\'enyi graphs that are more `homogeneous'). As a rule, random geometric graphs have larger clique numbers than Erd\H{o}s--R\'enyi graphs with the same average node degree. Indeed, in a geometric graph, a cluster of vertices that are sufficiently close to each other will automatically form a clique. This provides intuition that the planted clique problem in this case can be richer and will depend on the graph's parameters. 

In the current paper, we consider hard random geometric graphs (RGG) \cite{mcdiarmid2003random}, \cite{penrose2003random}, \cite{muller2008two}. In this model, the vertices' positions are chosen uniformly at random on a $d$-dimensional torus, where $d \geq 1$ is fixed, and a threshold distance $r:=r_n$ is set such that nodes are connected if and only if their distance (induced by the $L_2$ norm on the torus) is less than $r$. 

The main point of this paper is to show that the geometric structure of a graph allows for the detection of artificial structures (such as planted cliques) much more effectively than in Erd\H{o}s--R\'enyi graphs. It is useful to recall that the asymptotics of the clique number of the hard geometric graph on $n$ vertices is (see e.g., \cite[Theorem 1.2]{mcdiarmid2011chromatic} and \cite[Theorem 6.16]{penrose2003random}):
\begin{equation}\label{eq:natural_clique_number}
    \omega_n \sim 
    \begin{cases}
        O(1) \quad a.s., & \text{if $\mu \leq n^{-\varepsilon}$ for some $\varepsilon>0$},\\
        \frac{\log n}{\log\left( \frac{\log n}{\mu} \right)} \quad a.s., & \text{if $n^{-\varepsilon} \ll \mu \ll \log(n)$ for all $\varepsilon>0$},\\
        \mu f(t) \quad a.s., & \text{if $\mu / \log(n) \to t \in (0,\infty)$},\\
        \mu/2^d \quad a.s., & \text{if $\mu \gg \log(n)$},
    \end{cases}
\end{equation}
where $\mu$ is the expected degree of $G(n,r)$, and $f(t):= \frac{H_{+}^{-1}(2^d/t)}{2^d}$, where $H^{-1}_+$ is the inverse of the entropy $H(x) = 1-x+x \log(x)$ on the half-line $(1,\infty)$. Equation \eqref{eq:natural_clique_number} shows a major difference in the clique number among different regimes of RGG. In the dense and connectivity settings, the largest clique in the graph has size proportional to $\mu$; in sparse settings, it is super-linear in $\mu$. It might be tempting to conjecture that any artificially planted clique needs a size larger than the \textit{natural} clique number in RGG, to detect it. Although this might be the case for the Erd\H{o}s--R\'enyi models, in RGGs, it is possible to go far below the thresholds in \eqref{eq:natural_clique_number}, as we will show later.

Our main contributions are as follows. First, we study an algorithm based on vertex degrees (shortly, \textit{VD-algorithm}) that was introduced in \cite{kuvcera1995expected} for the Erd\H{o}s--R\'enyi model. We determine a recovery threshold for the VD-algorithm to recover a planted clique. Second, we introduce a simple polynomial algorithm based on common neighbours (we denote this algorithm by \textit{CN-algorithm}). We prove that the CN-algorithm solves the problem for hard random geometric graphs in a much richer regime. Surprisingly, the algorithm recovers even extremely small cliques, such as planted edges. Finally, we present practical experiments demonstrating that CN- and VD-algorithms are effective in practice. 

The paper is organized as follows. In Section \ref{sec:problem_description}, we introduce geometric random graphs, formulate the planted clique problem, and give a detailed description of VD- and CN-algorithms. Main results concerning the theoretical performance guarantees of the algorithms are collected in Section \ref{sec:algorithm_analysis}. Next, we provide in Section \ref{sec:numerical_experiments} numerical experiments as evidence supporting the theoretical results. Proof for the VD-algorithm and the CN-algorithm can be found respectively in Sections \ref{sec:VD_proof} and \ref{sec:CN_proof}. Lastly, we collect some thoughts on further research directions in Section \ref{sec:conclusion}.

\section{Problem statement}\label{sec:problem_description}

We will now state the problem more precisely. We use the common notation $G = (V,E)$ to indicate a simple graph, where $V$ is a set of vertices and $E$ is the set of edges.
Let $X_1, X_2, \ldots $ be independent random variables uniformly distributed in the unit $d$-dimensional cube $U = [0,1]^d$. To exclude boundary effects, we suppose that $U$ is equipped with the toroidal metric. Formally, for any position vectors $x = (x_1, ..., x_d), y = (y_1, ..., y_d)$, the distance on $U$ is given by 
$$
    d_T(x,y) = \left( \sum_{i=1}^d \min\left( |x_i - y_i|,\ 1 - |x_i - y_i| \right)^2 \right)^{1/2}. 
$$
The random geometric graph model is parameterized by the expected number of vertices, $n > 0$, and a connection probability function, $p_n$. Let us denote by $N_n$ a Poisson random variable with mean $n$. We consider the vertex set $V = \{1, \ldots, N_n\}$, where each vertex $i$ is associated with its position $X_i$. Any two vertices at distance $x$ are connected independently with the connection probability $p_n(x)$. In other words, 
$$
    p_n(x) = \mathbb P\bigl( \{i, j\} \in E \,|\,d_T(X_i, X_j) = x \bigr).  
$$
The model defined in this way is called a random geometric graph model, and will be denoted by $G(n, p_n)$. In what follows, we will focus on \textit{hard random geometric graphs} (RGG), where the connection probability $p_n$ has threshold form:
\begin{equation}
    p_n(x) = 1(x \leq r_n) \nonumber
\end{equation}
for some $r_n > 0$, and for technical reasons we will assume $r_n < 1/4$. In other words, in RGGs, each vertex has its sphere of influence of radius $r_n$, within which it forms connections. The resulting graph is denoted by $G(n,r_n)$. 

The average expected degree in $G(n,r_n)$ is denoted by $\mu_n$. Average degree and radius of the connectivity are linked through the equation $\mu_n = n \phi_d r_n^d$, where $\phi_d$ is the volume of the $d$-dimensional ball of radius 1.

We are now ready to describe the \textit{planted clique problem}, which is the object of the current investigation. Specify a positive integer $k_n \geq 2$, that can vary with the network size $n$. Consider the graph $G(n,r_n)$ and select randomly $k_n$ vertices $K_n = \{i_1, \ldots, i_{k_n}\}$ from $V$. Next, add all the missing edges between all pairs of vertices in $K_n$, forcing $K_n$ to form a clique. We denote the resulting graph by $G_{k_n}(n, r_n)$. Our goal is to recover the planted clique $K_n$. We focus on the conditions for the exact recovery, that is, we would like to build an algorithm that outputs a set $\hat{K}_n$ of size $k_n$ such that 
\begin{equation}
    \lim_{n \to \infty} \Prob{K_n = \hat{K}_n} = 1. \nonumber
\end{equation}

\paragraph*{Well-posedness of the problem}
In the process of planting a clique on the vertex set $K$ in $G(n, r_n)$, some edges between elements of $K$ may already be present. Therefore, only the missing edges are added. However, for certain parameter choices, the clique $K$ may already be complete in the original graph $G(n, r_n)$. In this case, $G(n, r_n) = G_{k_n}(n, r_n)$, making the problem ill-posed because it is impossible to distinguish an instance of the random geometric graph from the graph augmented by a planted clique. 
Consequently, we will exclude the ill-posed parameter settings defined by $\{k = O(1)\} \wedge \{\mu_n \propto n\}$ from our analysis because for this parameter regime, $\Prob{G(n, r_n) = G_{k_n}(n, r_n)} > 0$, and we might otherwise encounter recovery issues.

\paragraph*{Notation}
We now introduce some notations used throughout the paper. The degree of a vertex $i \in V$ in a graph is denoted by $Z_i = |N(i)|$, where $N(i)$ indicates the set of neighbours of $i$. The number of neighbours of $i$ in $U \subset V$ is instead denoted by $Z_i^{U} = |N(i) \cap U|$ (note that $i$ does not necessarily belong to $U$). Moreover, the number of common neighbours of vertices $i$ and $j$ is denoted by $Z_{ij} = |N(i) \cap N(j)|$. 
The maximum and minimum degrees in $G(n,r_n)$ are denoted by $\Delta_n$ and $\delta_n$ respectively.

The comparisons $f(n) \ll g(n)$, $f(n) \gg g(n)$ and $f(n) \sim g(n)$ are understood standardly: $f(n) = o(g(n))$, $g(n) = o(f(n))$ and $f(n) = g(n) (1 + o(1))$ respectively, as $n\to\infty$. For a graph $G$, the notation $\omega(G)$ denotes its clique number (the size of the largest clique). We say that a sequence of events $A_n$ happens almost surely (a.s.) if there exists $n_0 \in \mathbb N$ such that $\mathbb P(A_n) = 1$ for all $n > n_0$.  We say that a sequence of events $A_n$ occurs with high probability (or, equivalently, asymptotically almost surely) if $\lim_{n\to\infty} \mathbb P(A_n) = 1$.

\subsection{Vertex Degrees (VD) algorithm}

Since the addition of planted edges increases the degree of a vertex, the most natural and simplest way to detect a hidden clique is based on vertex degrees. This algorithm, although naive, gives a basic result for Erd\H{o}s--R\'enyi graphs and can be described as follows. First, we calculate the degrees of all the vertices. Then, we construct $\hat{K}^{(\text{VD})}$ by selecting the $k$ vertices with the largest degrees. We call this VD-algorithm, and its formal description is given in Algorithm \ref{algo:VD-algorithm}. 
\begin{algorithm}
	\KwInput{Graph $(V,E)$; size $k$ of the planted clique} 
	
	\KwOutput{A set of vertices}~\\
	
	\For{$i \in V$}
	{   
	    Count $Z_i = \deg(i)$
	}
	Select the largest $Z_{i_1}, \ldots, Z_{i_{k}}$
    \\
	\Return $\hat{K}^{(\text{VD})}=\{i_1, \ldots, i_{k}\}$
	\caption{VD-algorithm}
	\label{algo:VD-algorithm}
\end{algorithm}
The computational complexity of the VD-algorithm in $G_{k_n}(n, r_n)$ is $O_p(n(\log(k_n) + \mu_n))$ for computing the degrees of all vertices and finding top-$k_n$ degrees by the min-heap method~\cite{cormen2022introduction}. The main advantages of this algorithm are its rather low complexity and simple implementation.

The idea is that if $k_n$ is sufficiently large, then the algorithm returns exactly the set $K_n$ with high probability. This algorithm was first used for planted clique recovery in Erd\H{o}s--R\'enyi graphs (see \cite[Section 6, Theorem 6.1]{kuvcera1995expected}). 
\begin{theorem}[\cite{kuvcera1995expected}]\label{th-ER-vertex-degree-works}
    Let $G^{\text{ER}}_{k_n}(n,1/2)$ be an Erd\H{o}s--R\'enyi graph on $n$ vertices with a planted clique of size $k_n > C\sqrt{n\log n}$ with $C > 0$ large enough. Then, the VD-algorithm recovers the planted clique in $G^{\text{ER}}_{k_n}(n, 1/2)$ with high probability. 
\end{theorem}

\subsection{Common Neighbours (CN) algorithm}

Next, we introduce an algorithm based on the common neighbours of two adjacent vertices, Algorithm~\ref{algo:CN-algorithm}. It constructs the estimator $\hat{K}^{\text{(CN)}}$ as follows: we check all edges $(i,j)$; if $i$ and $j$ have in common exactly $k-2$ neighbours, we check whether $i,j$, and their common neighbours form a clique; if this is the case, the algorithm returns this $k$-clique. If such an edge $(i,j)$ is not found, it returns the empty set.
\begin{algorithm}
	\KwInput{Graph $(V,E)$; size $k$ of the planted clique}
	
	\KwOutput{A set of vertices}~\\
	
	\For{$(i,j) \in E$}
	{   
	    \CNStep \\
	    Let $Z_{ij} = N(i) \cap N(j)$ be the set of common neighbors of $i$ and $j$\\
	    \If{$|Z_{ij}| = k-2$}
	    {
	        \CliqueStep \\ 
	        \lIf{$Z_{ij}$ forms a clique of size $k-2$ }{\Return $\{i\} \cup \{j\} \cup Z_{ij}$} 
	    }
	}
	\Return $\hat{K}^{\text{(CN)}} = \varnothing$
	\caption{CN-algorithm}
	\label{algo:CN-algorithm}
\end{algorithm}
In $G_{k_n}(n,r_n)$, the worst case complexity of this algorithm is $O_p(\mu_n n (n + k_n^2))$, since for any fixed edge, finding common neighbours requires at most $n$ operations, and testing for a clique of size $k_n$ requires not more than $\binom{k_n}{2}$ operations. 

The intuition behind this algorithm is illustrated in Figure \ref{fig: common-neighbours-area-algo-desc}, which shows why the CN-algorithm does not confuse the planted clique with natural cliques in RGG. If two vertices $i,j\notin K_n$ form a natural edge, then they might have common neighbours in both regions $R_1$ and $R_2$, which are contained in the intersection of the two balls centered in $X_i, X_j$. Pairs of vertices located in opposite regions have a distance larger than $r_n$; thus, they are not naturally adjacent. In other words, if $i$ and $j$ have common neighbours in both $R_1$ and $R_2$, then their common neighbours cannot form a clique. Therefore, the algorithm is likely to skip the pair $(i,j)$.

The common neighbours approach has been widely used for link prediction in networks, through the definition of score-based methods that typically depend on the number of common neighbours of two vertices \cite{zhou2009predicting,liben2003link}. In this paper, we demonstrate that common neighbours in geometric networks are informative for other tasks, including anomaly detection. 

\begin{figure}[t]
    \centering
    \includegraphics[width = 0.6\textwidth]{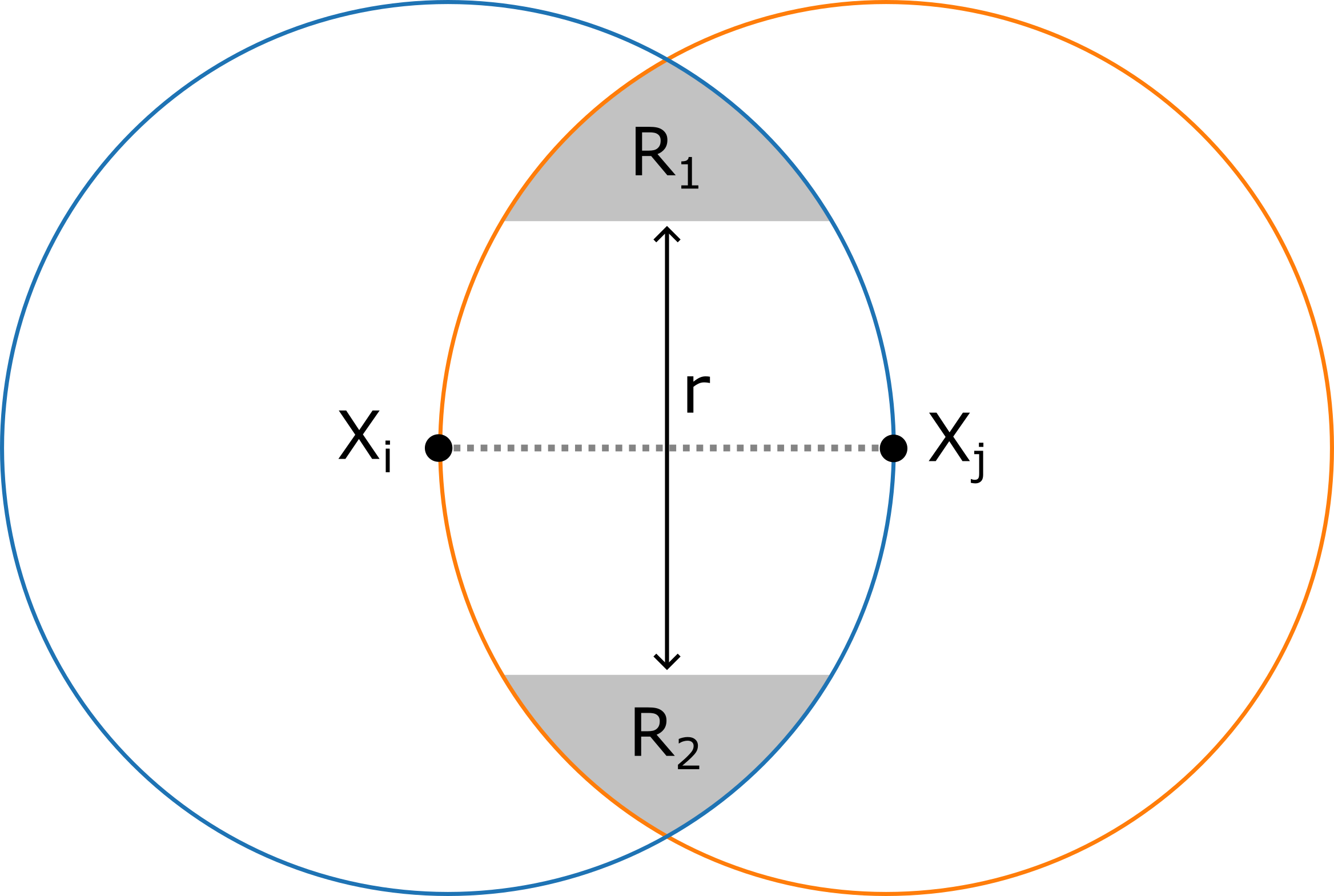}
    \caption{Representation in dimension $d=2$ of the regions $R_1$, $R_2$. Vertices with coordinates in $R_1$ cannot connect to vertices with coordinates in $R_2$, unless they are part of the planted clique.}
    \label{fig: common-neighbours-area-algo-desc}
\end{figure}

\section{Analysis of the algorithms}\label{sec:algorithm_analysis}

In this section, we summarize our results on VD- and CN-algorithms. Specifically, we retrieve the regimes of RGG under which exact recovery of the clique $K_n$ is possible, utilizing the estimators $\hat{K}_n^{\text{(VD)}}$ and $\hat{K}_n^{\text{(CN)}}$. 

\subsection{VD-algorithm}

Recall that $\mu_n=  n \phi_d r_n^d$ denotes the expected degree in $G(n, r_n)$, whereas $\Delta_n$ and $\delta_n$ stand for the maximum and minimum degrees, respectively. Assume that $\mu_n/\log(n) \to\alpha \in [0,\infty]$ when $n$ goes to infinity. From \cite{penrose2003random} and \cite{mcdiarmid2003random} we know that in the hard geometric graph conditioned on the number of vertices $N_n = n$,
\begin{equation}
    \label{eq:maximum_RGG_threshold}
    \Delta_n/T(n) \stackrel{p}{\rightarrow} 1 , \quad \text{ with }
    T(n) =  
    \begin{cases}
    \frac{\log(n)}{\log(\log(n)/\mu_n)} & \text{if $\alpha=0$ }, \\ 
    \mu H^{-1}_+(\alpha^{-1}) & \text{if $\alpha \in (0,\infty]$ },
    \end{cases}
\end{equation}
and
\begin{equation}\label{eq:minimum_RGG_threshold}
    \delta_n/t(n) \stackrel{p}{\rightarrow}  1, \quad \text{ with } 
    t(n) =  
    \begin{cases}
    0 & \text{if $\alpha \in [0, 1)$}, \\ 
    \mu H^{-1}_{-}(\alpha^{-1}) & \text{if $\alpha \in [1,\infty]$}.
    \end{cases}
\end{equation}
The functions $H^{-1}_+$ and $H^{-1}_-$ are the two branches of the inverse entropy function, $H(x) = 1-x+x \log(x)$. Using the Lambert $W$ function, for any $y>0$ we can write explicitly $H_+^{-1}(y) = \exp\left\{{W_0\left(\frac{y - 1}{e}\right) + 1}\right\}>1$ and $H_-^{-1}(y) = \exp\left\{{W_{-1}\left(\frac{y - 1}{e}\right) + 1}\right\}<1$. When $\alpha = \infty$, we use the convention $1/\infty = 0$; thus, in this case $H_{+}^{-1}(0) = H_{-}^{-1}(0) = 1$.

The results in Equations \eqref{eq:maximum_RGG_threshold}-\eqref{eq:minimum_RGG_threshold} extend to the random graph $G(n,r)$ defined in Section \ref{sec:problem_description}, where $N_n$ is Poisson with mean $n$. Indeed, with high probability $n - n^{2/3} < N_n < n + n^{2/3}$. Therefore,
\begin{align*}
    T(n-n^{2/3})<\Delta_n<T(n+n^{2/3}) \quad \text{w.h.p.}, \\
    t(n-n^{2/3})<\delta_n<t(n+n^{2/3}) \quad \text{w.h.p.},   
\end{align*} 
hence the maximum and minimum degrees squeeze onto the sequences $T(n)$ and $t(n)$ respectively.

We now state the main result for the VD-algorithm.
\begin{theorem}[VD-algorithm -- positive result]\label{thm:successVDalgo}
Given $\alpha = \lim_{n \to \infty} \mu_n /\log(n)$,
\begin{itemize}
    \item if $\alpha \in [0,\infty)$ (sparse or connectivity regime), then the VD-algorithm exactly recovers the planted clique in $G_{k_n}(n,r_n)$ as soon as $k_n > (1+\varepsilon)(T(n)-t(n))$, for any $\varepsilon > 0$ and $T(n),t(n)$ defined in \eqref{eq:maximum_RGG_threshold} and \eqref{eq:minimum_RGG_threshold};
    \item if $\alpha = \infty$ (dense regime), then the VD-algorithm exactly recovers the planted clique in $G_{k_n}(n,r_n)$ as soon as $k_n > \varepsilon \mu_n$, for any $\varepsilon > 0$.
\end{itemize}
\end{theorem}

The idea of the proof is quite simple. The output $\hat{K}_n^{\text{(VD)}}$ of Algorithm~\ref{algo:VD-algorithm} is the set of $k_n$ vertices with largest degree in $G_{k_n}(n,r_n)$. Therefore, we can quantify the probability for the algorithm to succeed by:
\begin{equation}
    \Prob{\hat{K}_n^{\text{(VD)}} = K} = \Prob{\min_{i \in K_n} Z_i > \max_{i \in V \setminus K_n} Z_i}. \nonumber
\end{equation}
In other words, for the VD-algorithm to correctly identify the clique, every clique vertex must have a degree strictly greater than any non-clique vertex. This occurs when the minimum degree within the clique, approximately $k_n + \delta_n$, exceeds the maximum degree outside it, $\Delta_n$. The resulting success condition, $k_n > \Delta_n - \delta_n$, follows directly from the asymptotic behavior of $\Delta_n$ and $\delta_n$ in \eqref{eq:maximum_RGG_threshold} and \eqref{eq:minimum_RGG_threshold}.

We do not have a complete result for the VD-algorithm, in that we cannot conclude that the algorithm fails or achieves only partial recovery in the entire region of the $\mu$--$k$ plane. Nonetheless, we can provide some negative results.

\begin{theorem}[VD-algorithm -- negative result]\label{thm:negative_VD_algo}
Suppose either:
\begin{enumerate}
    \item[(i)] $\lim_{n \to \infty} \mu_n / \log(n) = \alpha \in (0,\infty)$ and $k_n \le (1-\varepsilon) (T(n) - \mu_n)$, for $0<\varepsilon<1$;
    \item[(ii)] $\lim_{n \to \infty} \mu_n / \log(n) = \infty$, with $\mu_n=o(n)$ and $k_n=o(\sqrt{\mu_n})$.
\end{enumerate}
Then the VD-algorithm fails to recover the planted clique in $G_{k_n}(n, r_n)$ with high probability. 
\end{theorem}
Theorem \ref{thm:negative_VD_algo} explains that small planted cliques cannot be recovered from solely ordering vertex degrees in the connectivity and super-connectivity regime. In the next section, we show that the CN-algorithm is more powerful in this sense.

The proofs of Theorems~\ref{thm:successVDalgo} and \ref{thm:negative_VD_algo} are given in Section \ref{sec:VD_proof}.

\subsection{CN-algorithm}

The output $\hat{K}_n^{\text{(CN)}}$ of Algorithm~\ref{algo:CN-algorithm} exactly matches the planted clique $K_n$ when: (a) there exists at least one edge $(i,j)$ of the planted clique such that $|Z_{ij}| = k_n-2$ and $Z_{ij}$ forms a clique; {\it and} (b) for all other edges $(i,j)$, either $|Z_{ij}| \not = k_n-2$ or $Z_{ij}$ does not form a clique. Thus, it is possible to explicitly compute the success probability of the CN-algorithm by analysing these two conditions separately.
The next theorem summarizes the regimes under which the CN-algorithm works with high probability for the graph density $\mu_n$ and the clique size $k_n$.

\begin{theorem}[CN algorithm -- positive result]\label{thm:successCNalgo}
The CN-algorithm succeeds with high probability if 
\begin{equation}\label{eq:cond2successCN}
    k_n\le \alpha n, \alpha \in (0,1), \quad \mu_n n e^{- c_{1,d} \mu_n} =o(1)
\end{equation}
or
\begin{equation}\label{eq:cond3successCN}  
    \begin{cases}
    \displaystyle \frac{\mu_n n}{2} \frac{(c_{2,d} \mu_n)^{k_n-2} e^{- c_{2,d} \mu_n}}{(k-2)!} =o(1) &\text{if $k_n-2 \leq c_{2,d} \mu_n$}, \\
    \displaystyle \frac{n}{2} \frac{\mu_n^{k_n-1} e^{- \mu_n}}{(k_n-2)!} = o(1) &\text{if $k_n-2 \geq \mu_n$ and $k_n = o(n/\mu_n)$}, 
    \end{cases}
\end{equation}
for some constants $c_{1,d} >0, c_{2,d}>0$. 
\end{theorem}

To get an explicit description of the region where $\hat{K}_n^{\text{(CN)}}$ exactly matches $K_n$, the inequalities in \eqref{eq:cond2successCN} can be solved using the Lambert $W$ function. In particular, it is easily checked that $\mu_n n e^{-c_{1,d} \mu_n} = o(1)$ if and only if $\mu_n \gg -W_{-1}(-c_{1,d}/n)/c_{1,d} = (\log n + \log \log n)/c_{1,d} + O(1)$, using the asymptotic expansion for the Lambert $W$ function (see \cite{corless1996lambert}). In particular, the CN-algorithm correctly identifies a planted clique of any size if $\mu_n$ is large enough, but sublinear in the graph size $n$.

We observe that the CN-algorithm is effective in certain special cases. Importantly, it recovers small cliques. Indeed, take $k=2$ (planted edge) and $\mu_n \ge (1+\epsilon)\log(n)$. This corresponds to the case $k_n-2\le c_{2,d}\mu_n$  in \eqref{eq:cond3successCN}, and it is easy to check that the condition is satisfied. This applies to any finite $k_n$. This is in contrast to the VD-algorithm. We also note that the CN-algorithm recovers large cliques in the sparse regime and the connectivity regime; it is sufficient that  $k_n\log k_n\gg \log (n)$ and $k_n \gg \mu_n$.\\

\begin{figure}
    \begin{subfigure}[t]{0.48\linewidth}
        \includegraphics[width=\linewidth]{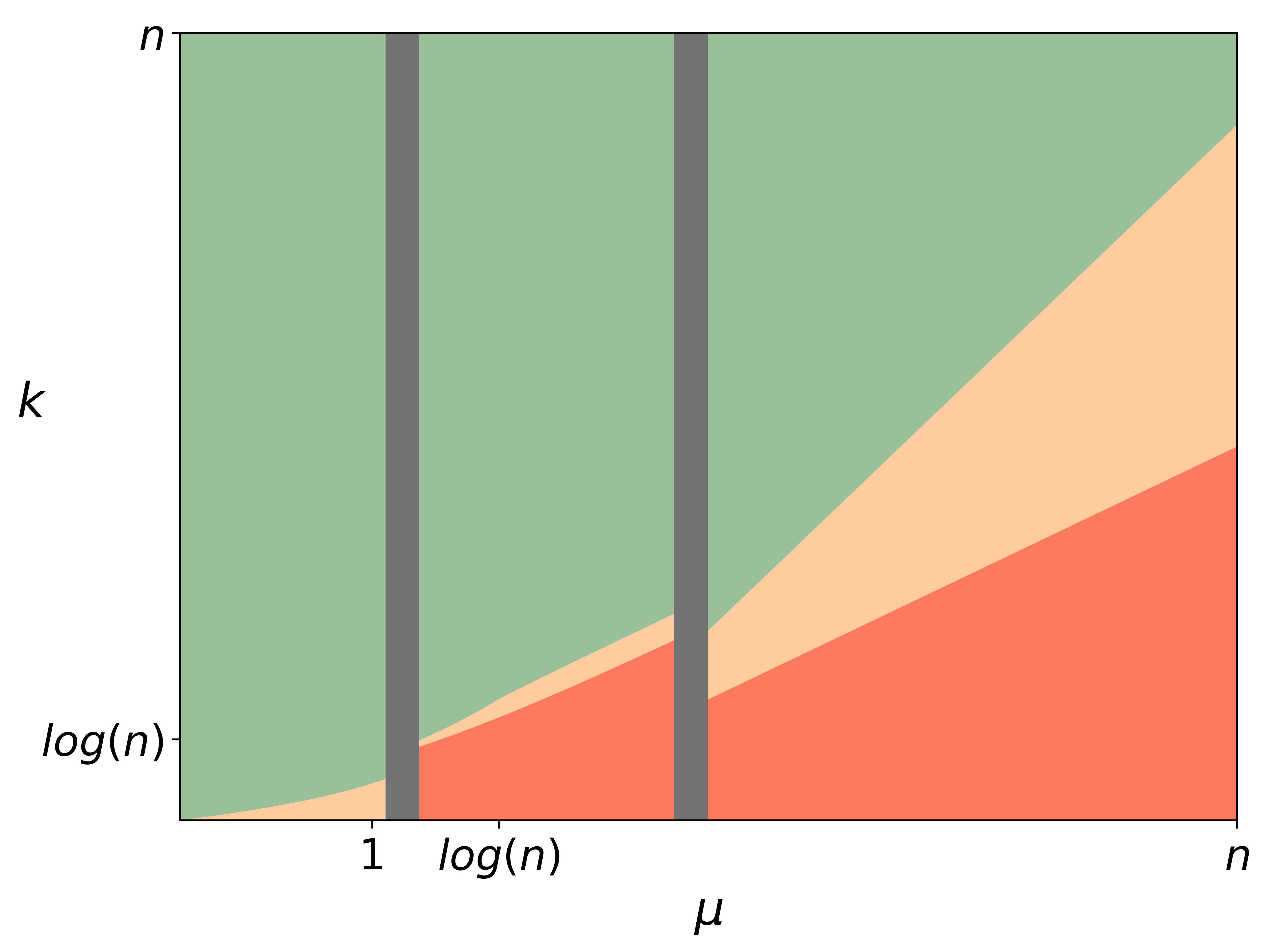}
        \caption{VD-algorithm}
        \label{fig:success_VD_algo}
    \end{subfigure}    
    \hfill
    \begin{subfigure}[t]{0.48\linewidth}
        \includegraphics[width=\linewidth]{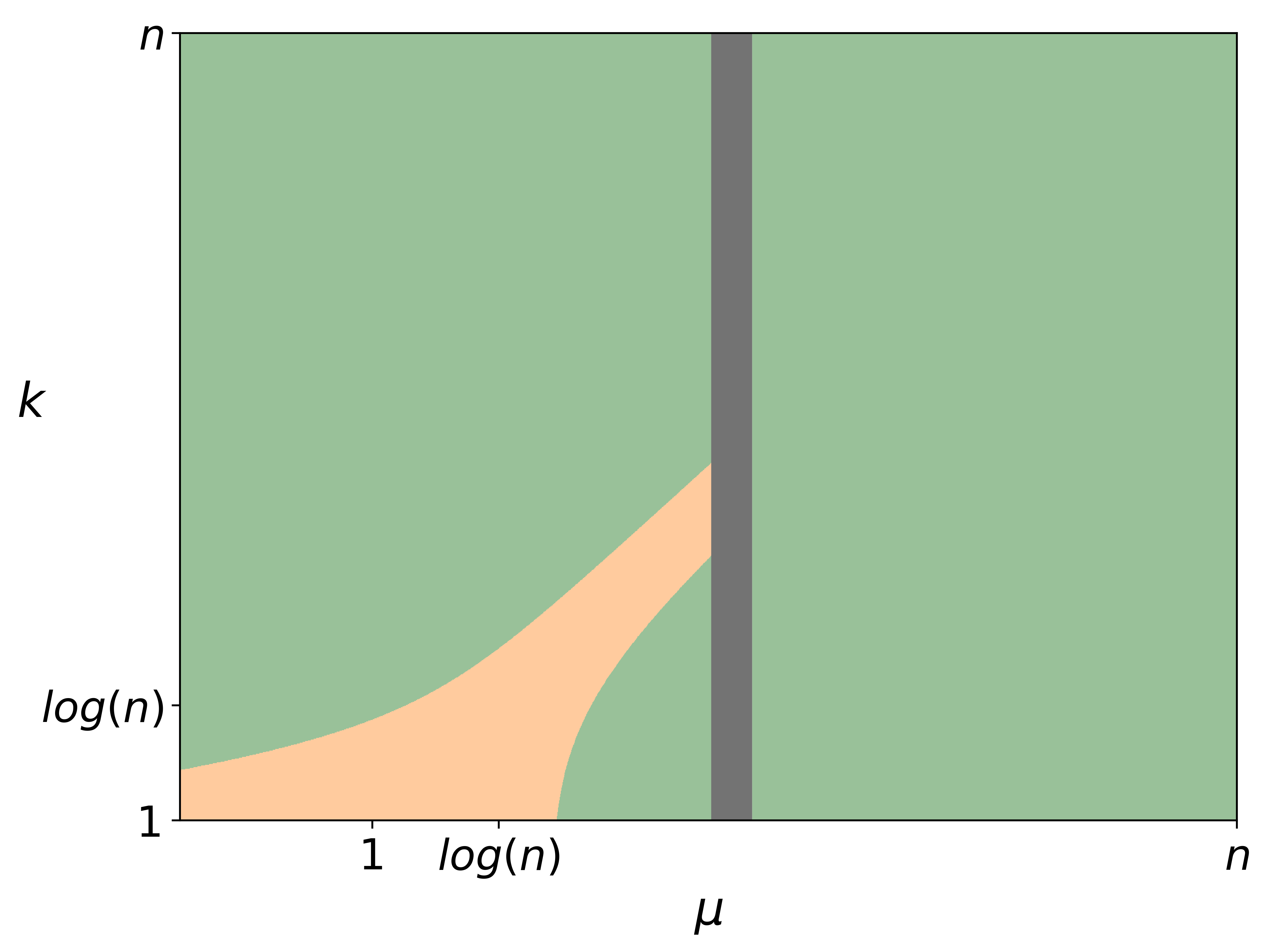}
        \caption{CN-algorithm}
        \label{fig:success_CN_algo}
    \end{subfigure}
    \vspace{8pt}
    \caption{Qualitative representation of recovery results, in the $\mu$--$k$ plane, according to Theorems \ref{thm:successVDalgo}, \ref{thm:negative_VD_algo} and \ref{thm:successCNalgo}. The pictures are shown for $n=10^9$, on a log-log scale. Green areas: exact recovery of the planted clique. Red areas: VD algorithm fails. Beige area: no theoretical guarantees are currently available. Gray vertical bands: boundaries between the sparse, connectivity, and dense regimes of geometric random graphs. The gray vertical bands separate different regimes of the average degree $\mu$: in Figure~\ref{thm:successVDalgo}, the gray band separates the two regions where condition \eqref{eq:cond2successCN} is satisfied or not.}
    \label{fig:success_comparison}
\end{figure}

Figure \ref{fig:success_comparison} gives a qualitative picture of the region in the $\mu$--$k$ plane where the VD- and CN-algorithms succeed or fail with high probability. Green areas indicate parameter ranges where $\hat{K}^{\text{(VD)}}$ and $\hat{K}^{\text{(CN)}}$ exactly recover the planted clique $K$ in random geometric graphs. The red region marks where the VD-algorithm fails. Beige regions correspond to parameter ranges for which no theoretical guarantees are currently available. The gray vertical bands separate different regimes of the average degree $\mu$: in Figure~\ref{thm:successVDalgo}, they distinguish the sparse, connectivity, and dense regimes of geometric random graphs; in Figure~\ref{fig:success_CN_algo}, they indicate the two regions where condition \eqref{eq:cond2successCN} is satisfied or not.

\section{Numerical experiments}\label{sec:numerical_experiments}

From a theoretical standpoint, Theorems \ref{thm:successVDalgo}, \ref{thm:negative_VD_algo}, and \ref{thm:successCNalgo} show that asymptotically the CN-algorithm recovers the planted clique over a wider range of random geometric graph parameters than the VD-algorithm. In particular, as the average degree $\mu_n$ increases, the VD-algorithm fails to recover small planted cliques (see Figure \ref{fig:success_VD_algo}). By contrast, Figure \ref{fig:success_CN_algo} shows that the CN-algorithm can successfully recover planted cliques of any size order (including constant size) as $\mu_n$ grows. Below, we demonstrate numerically that the common-neighbours approach outperforms vertex-degree ordering not only asymptotically (as already proved) but also on finite-size samples of $G_k(n,r)$, making it a practical method for finite graphs.

\begin{figure}
\vspace{10pt}\centering
\begin{subfigure}{0.6\linewidth}
\centering
\includegraphics[width=\linewidth]{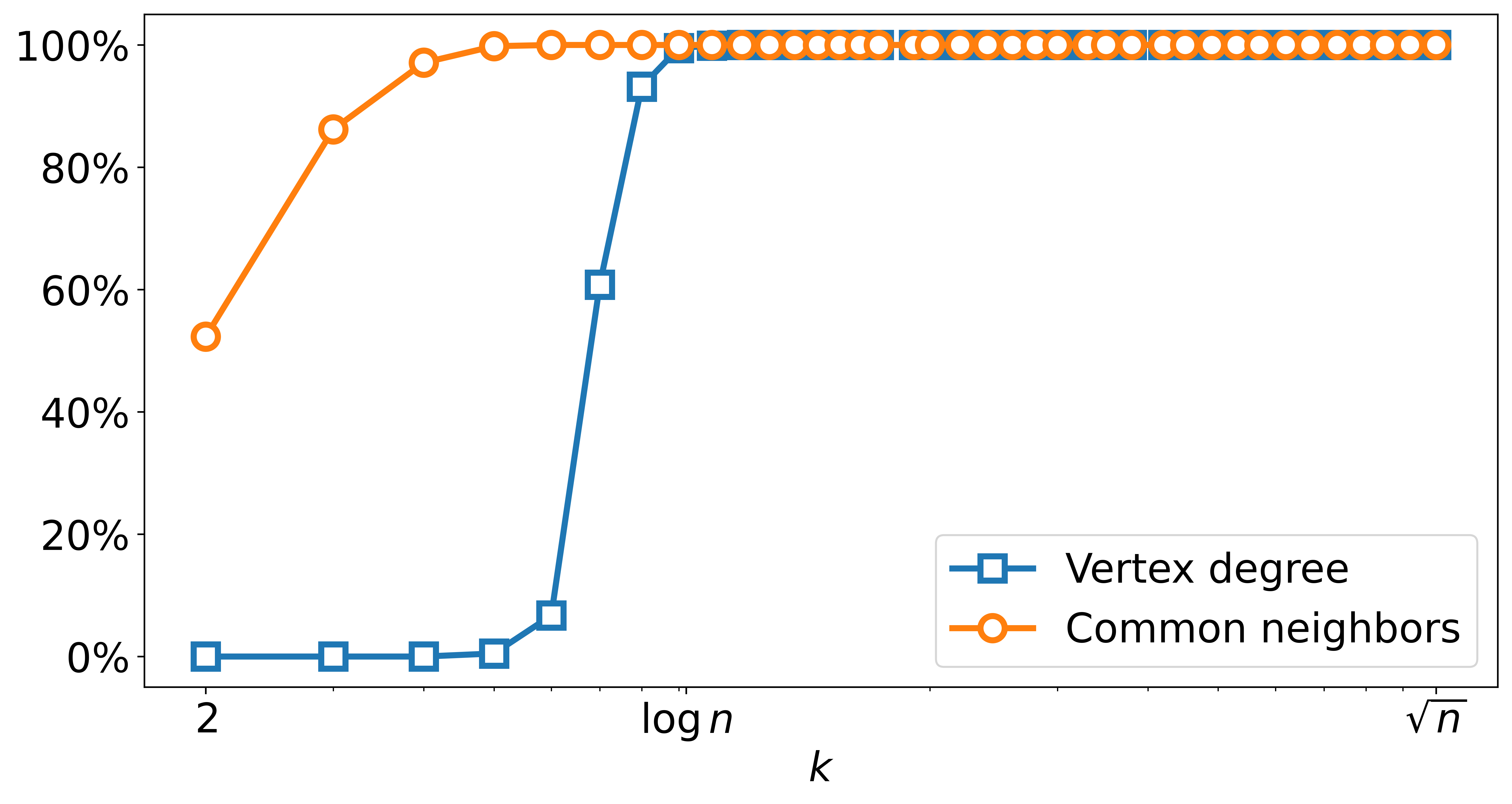}
\caption{$\mu = 1$}\vspace{10pt}
\label{fig:compare_mu1}
\end{subfigure}
\begin{subfigure}{0.6\linewidth}
\centering
\includegraphics[width=\linewidth]{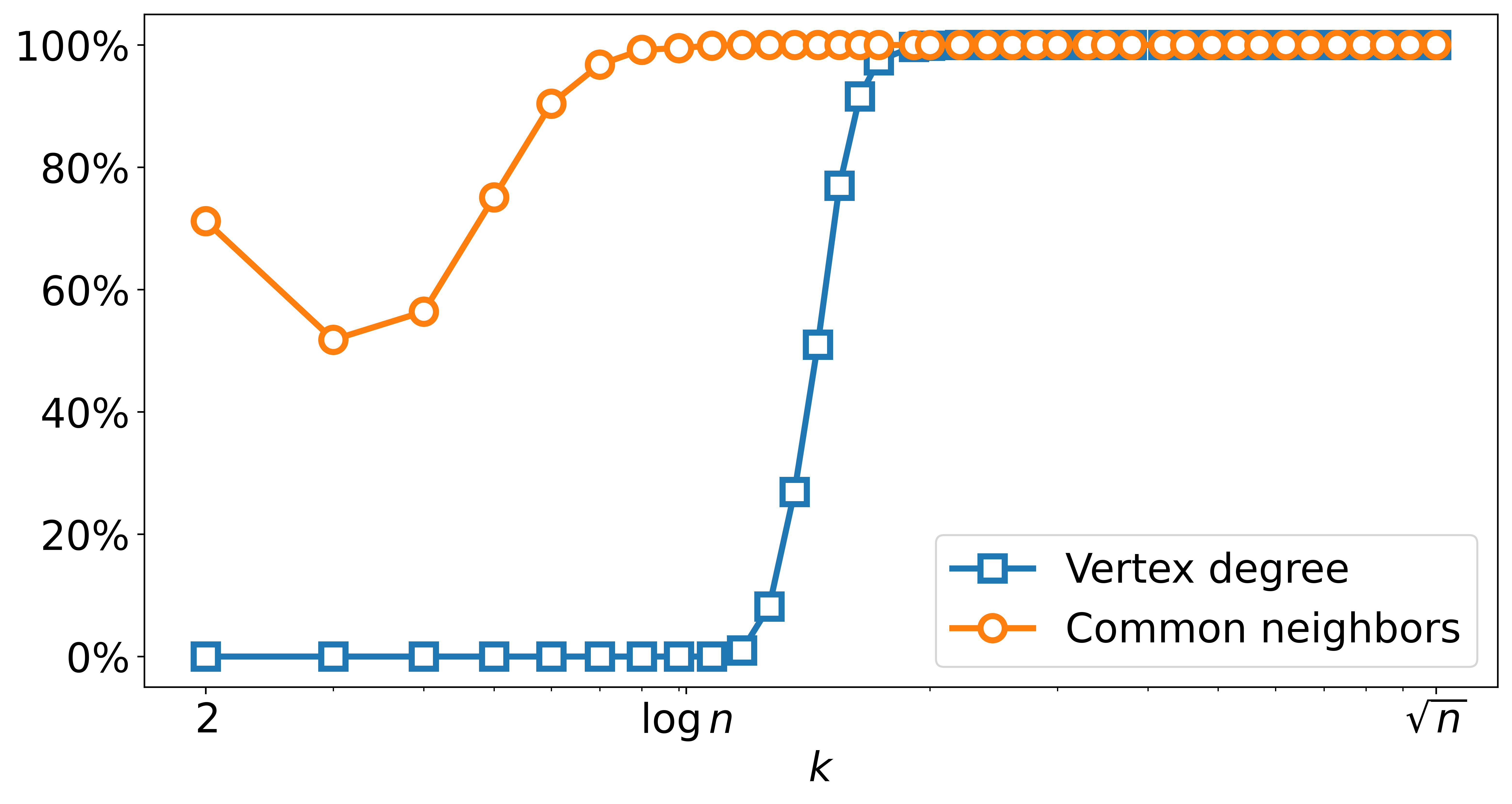}
\caption{$\mu = 5$}\vspace{10pt}
\label{fig:compare_mu5}
\end{subfigure}
\begin{subfigure}{0.6\linewidth}
\centering
\includegraphics[width=\linewidth]{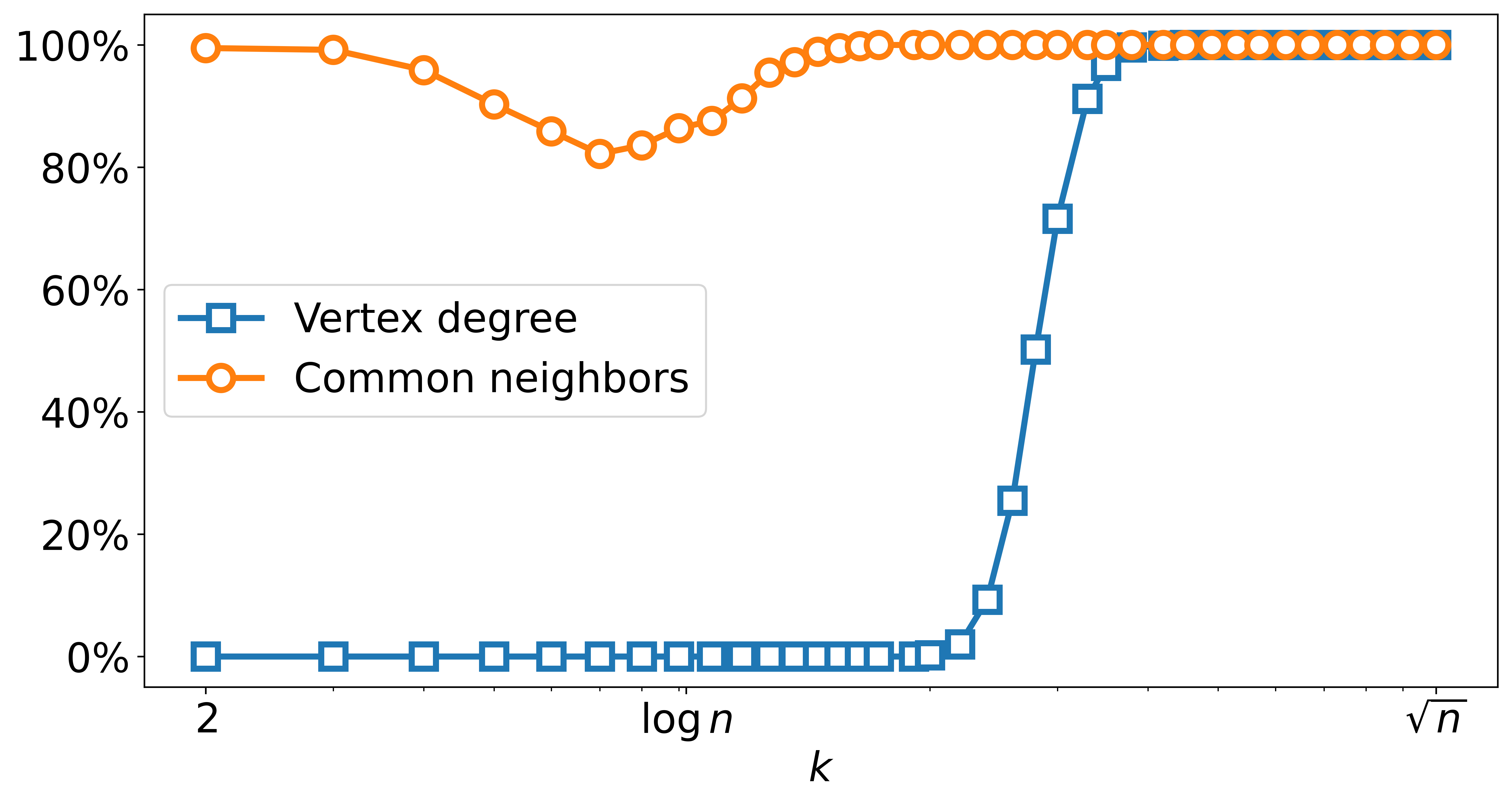}
\caption{$\mu = 20$}\vspace{10pt}
\label{fig:compare_mu20}
\end{subfigure}
\vspace{8pt}
\caption{Percentage of runs (out of 1000) in which the VD- and CN-algorithms returned $K$ correctly for $G_{k}(n,r)$ with $n = 10^4$, $d=2$.}
\label{fig:compare}
\end{figure}

Figure \ref{fig:compare} reports the success rate (over 1000 independent samples of $G_k(n,r)$) of the VD- and CN-algorithms as the planted clique size $k$ increases, for three values of the average degree: $\mu=1$ (Figure \ref{fig:compare_mu1}), $\mu=5$ (Figure \ref{fig:compare_mu5}), and $\mu=20$ (Figure \ref{fig:compare_mu20}). For reference we used $n=10^4$, so $\log(n)\approx 9.21$. The observed behavior agrees with our theoretical predictions:
\begin{itemize}
\item The VD-algorithm performs well for small average degree but deteriorates as $\mu$ grows: it requires the planted clique to be sufficiently large to be identified correctly.
\item The CN-algorithm achieves recovery on almost the entire tested parameter range; for larger $\mu_n$, it recovers any planted clique size (even planted edges).
\item In all settings tested, the CN-algorithm outperforms the VD-algorithm.
\end{itemize}

\section{Proofs for the VD-algorithm}\label{sec:VD_proof}
For simplicity of notation, from now on, we will write $\mu = \mu_n$, $k = k_n$, and $r = r_n$ since it is clear that all these quantities scale with the size $n$.

\begin{proof}[Proof of Theorem \ref{thm:successVDalgo}]
Our aim is to find under which regimes for $\mu$ and $k$ in $G_k(n,r)$
\begin{equation}\label{eq:condition_VD_algo}
    \lim_{n \to \infty} \Prob{\min_{i \in K} Z_i > \max_{i \in V \setminus K} Z_i} = 1.
\end{equation}
First, observe that 
\[
\max_{i \in V \setminus K} Z_i\le \Delta_n.
\]
Moreover, for each clique vertex $i \in K$, we have $Z_i = (k-1) + Z_i^{V \setminus K}$, thus 
\[
\min_{i \in K} Z_i \ge k - 1 + \min_{i \in K} Z_i^{V \setminus K} \ge k - 1 + \delta_{n-k}.
\]
Then, condition \eqref{eq:condition_VD_algo} is verified if $k > \Delta_n - \delta_{n-k} + 1$ holds with high probability, in the large $n$ limit.
At this point, we can examine the cases separately for finite and infinite $\alpha$. If $\alpha \in [0, \infty)$ (connectivity or sub-connectivity regimes), observe that 
\[\frac{\Delta_n - \delta_{n-k}}{T(n) - t(n-k)} \stackrel{p}{\rightarrow} 1.\]
In particular: if $k = o(n)$, then $t(n) = (1+o(1))t(n-k)$ and condition \eqref{eq:condition_VD_algo} is verified as long as $k > (1+\varepsilon) (T(n) - t(n))$, for some $\varepsilon > 0$; if $k = \Theta(n)$, condition \eqref{eq:condition_VD_algo} is trivially satisfied.
Instead, when $\alpha = \infty$ we have $H_{-}^{-1}(\alpha^{-1}) = H_{+}^{-1}(\alpha^{-1}) = 1$, thus
\[\lim_{n \to \infty} \frac{\Delta_n - \delta_{n-k}}{\mu} = 0, \quad \text{ in probability}.\]
Therefore, \eqref{eq:condition_VD_algo} is achieved by choosing $k = \varepsilon \mu$, for any $\varepsilon >0$.

Consequently, as soon as $k > (1+\varepsilon) T(n)$, $\min_{i \in K} Z_i > \max_{i \in V \setminus K} Z_i$ with high probability, and the VD-algorithm correctly recovers the planted clique.
\end{proof}

Interestingly, despite a simple proof, it is not easy to relax the conditions of Theorem~\ref{thm:successVDalgo}. Indeed, consider $k=o(n)$. Then the following lemma holds. 
\begin{lemma}\label{lem:max}
Consider $G_k(n,r)$, with $k=o(n)$. Then 
\begin{equation} \nonumber
    \label{eq:maxV-K}
    \lim_{n\to\infty}\Prob{\max_{i\in V\setminus K}Z_i=\Delta_n}=1.
\end{equation}
where $\Delta_n$ denotes the maximum degree in $G(n,r)$.
\end{lemma}
\begin{proof} Let $G(n,r)$ be a random geometric graph and $G_k(n,r)$ its modified version where a clique is planted on $k$ randomly chosen vertices. Without loss of generality, assume that vertex~$1$ has maximal degree in $G(n,r)$. We have $1\in \{\text{argmax}_{i\in V\setminus K} Z_i\}$ if and only if $1 \not \in K$. This latter event is independent of $Z_1$ and occurs with probability $1-k/n$. Since $k=o(n)$, this probability goes to one as $n\to\infty$.
\end{proof}

From Lemma $\ref{lem:max}$ we infer that if for all clique vertices $i \in K$ the relation $Z_i^{V \setminus K} + k > \Delta_n$ holds with high probability, then the VD-algorithm works. Therefore, it is sufficient to show that $\min_{i\in K}|\{j\in N(i), j\notin K\}| + k > \Delta_n$. However, it is not immediate to compute this minimum as the degree of the clique vertices may be dependent. We can partially overcome these difficulties in the next Lemma~\ref{lem:independent_vertices_K}, showing that for sufficiently small $k$ the dependency vanishes. 

\begin{lemma}\label{lem:independent_vertices_K}
Let $k=o\left(\sqrt{\frac{n}{\mu}}\right)$. Then the degrees of the clique vertices are asymptotically independent. 
\end{lemma}

\begin{proof}
    The degrees of the clique vertices are independent if and only if there is no overlap between the balls of radius $r$ centered at the vertices. As in the birthday paradox,
\[
    \Prob{\mbox{balls centered at clique vertices do not overlap}}
    =\prod_{l=1}^{k-1}\left(1-l\phi_d r^d\right).
\]
Assuming $k^2r^d\to 0$, for all $n$ large enough we have $l\phi_d r^d\le 1$ for every $1\le l\le k-1$. Therefore, 
\[
    \Prob{\mbox{balls centered at clique vertices do not overlap}}
     \ge 1-\sum_{l=1}^{k-1} l\phi_d r^d
    =1-\phi_d r^d\frac{k(k-1)}{2}.  
\]
If $k^2r^d\to 0$, then $\phi_d r^d\frac{k(k-1)}{2} = o(1)$, thus the probability that the balls overlap tends to zero, and therefore the degrees of the clique vertices are asymptotically independent. The condition $k^2 r^d \to 0$ is equivalent to
\[
    k=o\left(\sqrt{\frac{n}{\mu}}\right),
\]
since $\mu=n\phi_d r^d$.
\end{proof}

\begin{proof}[Proof of Theorem \ref{thm:negative_VD_algo}]
(i) Since the balls of radius $r$ centered at the clique vertices do not overlap w.h.p., it is sufficient to prove that
\begin{align*}
&\PP\left(\min_{i \in K} Z_i \le \max_{i \in V \backslash K} Z_i\right)\to 1.
\end{align*}
From \eqref{eq:maximum_RGG_threshold} and Lemma~\ref{lem:max} we know that $\max_{i \in V \setminus K} Z_i = T(n) (1+o_{\mathbb{P}}(1))$. 
Write $T(n):= \mu h$, with $h = H_{+}^{-1}(\alpha^{-1})>1$ as in Equation \eqref{eq:maximum_RGG_threshold}. Therefore, it is sufficient to prove that
\begin{equation}
    \Prob{\exists i \in K \colon Z_i^{V\backslash K} \leq T(n) (1+o_{\mathbb{P}}(1)) - k} \to 1. \label{eq:second_condition}
\end{equation}
At this point, we can rewrite
\begin{align}
    &\Prob{\exists i \in K \colon Z_i^{V\backslash K} \leq T(n)(1+o_{\mathbb{P}}(1)) - k} = 1 - \Prob{\forall i \in K \colon Z_i^{V\backslash K} > T(n) - k} \nonumber \\
    &\qquad = 1 - \Prob{Z_i^{V\backslash K} > T(n) (1+o_{\mathbb{P}}(1)) - k}^k(1+o(1)), \nonumber
\end{align}
where in the last step we used Lemma \ref{lem:independent_vertices_K}.

Applying a Chernoff bound to the Poisson random variable $Z_i$ for $n$ large, we obtain
\begin{align}
    \Prob{Z_i^{V\backslash K} > T(n) (1+o_{\mathbb{P}}(1)) - k}^k &\leq \Prob{Z_i^{V\backslash K} > h' \mu}^k + \PP\left(o_\PP(1)<-\delta(h-h')\right) \nonumber\\ 
    &\le e^{-\mu k(h' \log h' - h' + 1)} +  o(1) = o(1) \nonumber
\end{align}
for some $1<h'<1+\delta(h-1)$ (recall that $k\le (1-\delta)(T(n)-\mu)=(1-\delta)(h-1)\mu$; therefore, $T(n)-k=h\mu-k\ge \mu+ \delta(h-1)\mu$). Thus, condition \eqref{eq:second_condition} is satisfied.\\

(ii) We will prove that, for any $\varepsilon > 0$,
$$ \limsup_{n \to \infty} \Prob{\max_{i \in V\setminus K} Z_i > \min_{i \in K} Z_i}  \geq 1-\varepsilon. $$
Let $M>0$. We can tessellate the torus with $M^d$ cubes of side length $1/M$ each. To get the desired result, we will proceed in two steps: first, we show that every cube has at least one node from $V \setminus K$ such that the ball of radius centered in this node is completely contained in the cube; second, for small $k$, we show that some of the balls previously constructed may naturally contain an \textit{excess-}$k$ amount of nodes.

Formally, partition the cube $[0,1]^d$ into cubes $C_1,C_2,...$ of side length $1/M$, indexed in the set $\mathcal{M}$ with cardinality $M^d$. Each cube $C_i$ completely contains a ball $B_i$ centered in $v_i$ if and only if the node $v_i$ is contained in the \textit{core} $Q_i$ of the cube. In particular, the core $Q_i$ is a smaller $d$-dimensional cube of side length $(1/M - 2r)$. Denote by $\mathcal{Q}_i$ the set of vertices in $V \setminus K$ whose positions are contained in $Q_i$. Then,
\begin{align}
    \Prob{\bigcap_{i \in \mathcal{M}} \{\mathcal{Q}_i \not = \varnothing \}} &= 1 - \Prob{\bigcup_{i \in \mathcal{M}} \{\mathcal{Q}_i \not = \varnothing \}} \nonumber \\
    &\geq 1 - M^d \Prob{\{\mathcal{Q}_1 \not = \varnothing \}}. \nonumber
\end{align}
Since $N_n-k > \alpha n$ for some $\alpha \in (0,1)$, with high probability, we obtain $\Prob{\{\mathcal{Q}_1 \not = \varnothing \}} \le e^{-\alpha n(1/M-2r)^d}+o(1)$. Then, for any fixed $M$, the probability that all cores contain at least one vertex that is not part of the planted clique converges to 1, as $n$ grows. In other words, with high probability, each of the cubes $\{C_i\}_{i \in \mathcal{M}}$ contains a vertex $v_i \in V \setminus K$ such that its ball of radius $r$ is completely contained in $C_i$ and these balls are all disjoint.

Since the balls around vertices $\{v_i\}_{i \in \mathcal{M}}$ do not overlap, the degrees of these vertices are i.i.d. $\text{Poisson}(\mu)$. Moreover, the minimum degree among the planted clique vertices is stochastically dominated by the number of nodes contained in a ball of radius $r$ augmented by $k$, that is, $\min_{i \in K} Z_i \prec Y + k$, where $Y \sim \text{Poisson}(\mu)$.

The VD-algorithm fails to output the correct planted clique whenever $\max_{i \in V\setminus K} Z_i > \min_{i \in K} Z_i$. Computing the probability of the complementary event yields
\begin{align}
    \Prob{\max_{i \in V\setminus K} Z_i \le \min_{i \in K} Z_i} &\leq \Prob{\max_{i \in V\setminus K} Z_i \le Y + k} \nonumber \\
    &\leq \Prob{Z_{i} \le Y + k, \forall i \in \mathcal{M}}\nonumber \\
    &= \sum_{y=0}^{\infty} \Prob{Y = y} \Prob{Z \le y+k}^{M^d} \nonumber
\end{align}
where $Y,Z \sim \text{Poisson}(\mu)$. Since $\mu \to \infty$, in the large $n$ limit the random variable $\text{Poisson}(\mu)$ is approximated by $\mathcal{N}(\mu,\mu)$. Then, substituting $y = \mu + s \sqrt{\mu}$,
\begin{align}
   \liminf_{n \to \infty} \Prob{\max_{i \in V\setminus K} Z_i < \min_{i \in K} Z_i} &\leq \liminf_{n \to \infty} \int_{-\infty}^{\infty} \varphi(s) \Phi\left(s + \frac{k}{\sqrt{\mu}}\right)^{M^d} ds \nonumber
\end{align}
where $\varphi(\cdot)$ and $\Phi(\cdot)$ are the probability and cumulative distribution functions of the standard normal distribution. Lastly, set $\delta := k/\sqrt{\mu} = o(1)$. A first-order approximation yields $\Phi(s + \delta) = \Phi(s) + \delta\varphi(s) + O(\delta^2)$. Therefore,
\begin{equation*}
    \int_{-\infty}^{\infty} \varphi(s) \Phi\left(s + \delta\right)^{M^d} ds = \int_{-\infty}^{\infty} \varphi(s) \Phi(s)^{M^d} ds + \delta M^d \int_{-\infty}^{\infty} \varphi(s)^2 \Phi(s)^{M^d-1} ds + O(\delta^2).
\end{equation*}
Since $\varphi(s) = \Phi(s)'$, the first integral in the latter equation equals $\frac{1}{M^d + 1}$; then, bounding the second integral with 1 and putting the pieces together, we obtain:
\begin{equation}
    \liminf_{n \to \infty} \Prob{\max_{i \in V\setminus K} Z_i < \min_{i \in K} Z_i} \leq \frac{1}{M^d + 1}. \nonumber
\end{equation}
Finally, by setting $M^d = 1/\varepsilon -1$, the desired result is achieved.
\end{proof}

\section{Proofs for CN-algorithm}\label{sec:CN_proof}

The CN algorithm iterates over all edges and stops whenever it encounters an edge with $k-2$ common neighbours forming a clique. There are two scenarios where the algorithm fails to identify the $k$-clique:
\begin{itemize}
    \item[(a)] All pairs $(i,j) \in K^2$ have at least one common neighbour outside $K$.
    \item[(b)] There exists an edge outside the clique, $(i,j) \in E \setminus K^2$, such that the number of common neighbours is $|Z_{ij}| = k-2$ and they form a clique.
\end{itemize}
Therefore, we can express the failure of the CN-algorithm in terms of the following events:
\begin{align}
    \mathcal{A} &= \mathcal{A}_n=\left\{ \forall (i,j) \in K^2, |Z_{ij}| > k-2 \right\}, \nonumber \\
    \mathcal{B}_1 &=  \mathcal{B}_{n,1}=\left\{ \exists (i,j) \in E \setminus K^2 : N_i\cap N_j \textit{ form a clique} \right\},\nonumber\\
    \mathcal{B}_2 &= \mathcal{B}_{n,2}=\left\{ \exists (i,j) \in E \setminus K^2 : |Z_{ij}| = k-2 \right\}. \nonumber
\end{align}
The probability that the algorithm fails is upper bounded by the probability of the event $\mathcal{A}\cup (\mathcal{B}_1\cap \mathcal{B}_2)$, and
\begin{equation}
\label{eq:failure-CN-AB12}
\PP(\mathcal{A}\cup (\mathcal{B}_1\cap \mathcal{B}_2))\le \PP(\mathcal{A}) + \PP(\mathcal{B}_1\cap \mathcal{B}_2)\le \PP(\mathcal{A}) +\min\{\PP(\mathcal{B}_1),\PP(\mathcal{B}_2)\}.\end{equation}
We compute the probability of $\mathcal{A},\mathcal{B}_1,\mathcal{B}_2$ separately, and after that conclude under which conditions on $k = k(n)$ and $\mu = \mu(n)$ the algorithm succeeds, with high probability. We remind that in our assumptions, the connection radius is set to be $r \in (0, 1/4)$ for RGG, and the planted clique problem is assumed to be well-posed, as stated in Section \ref{sec:problem_description}.
\begin{lemma}
\label{prop:CNfailA}
In $G_k(n,r)$,
\begin{equation}\label{eq:CNfailA}
    \Prob{\mathcal{A}} \to 0 \quad \text{as $n \to \infty$}. \nonumber
\end{equation}
\end{lemma}

\begin{proof}
Observe that for any edge $(i,j) \in K^2$ the number of common neighbours is $|Z_{ij}| \geq k-2$, as all the other vertices in the clique are common neighbours of $i$ and $j$. However, $|Z_{ij}| > k-2$ if and only if $i,j$ do share at least one common neighbour from the vertex set $V \setminus K$. A necessary condition for such an event is $d(X_i, X_j) < 2r$. Therefore:
\begin{align}
    \Prob{\mathcal{A}} &= \Prob{\forall i,j \in K : |Z_{ij}| > k-2} \nonumber \\
    &= \Prob{\forall i,j \in K : N(i) \cap N(j) \cap ( V \setminus K ) \not = \varnothing} \nonumber \\
    &\leq \Prob{\forall i,j \in K : d(X_i, X_j) < 2r} \label{eq:bound_A},
\end{align}
We now distinguish two cases: when $k=O(1)$, or it grows with $n$.\\

\textit{Case $k = O(1)$}: Pivoting around a selected vertex $i \in K$, a necessary condition for the event in \eqref{eq:bound_A} to happen is that all other vertices $j \in K$ have position inside a ball centered in $X_i$ of radius $2r$. That is, 
\begin{align}
    \Prob{\forall i,j \in K : d(X_i,X_j) < 2r} &\leq \Prob{\forall j \in K : X_j \in B(X_i,2r)} \nonumber \\
    &= ( 2^d \phi_d r^d)^{k-1} = (2^d \mu/n)^{k-1} \nonumber,
\end{align}
and the latter quantity converges to 0 as $n$ goes to infinity unless $\mu \propto n$, that is, unless the problem is ill-posed.\\

\textit{Case $k \gg 1$}: From the technical assumption $0< r < 1/4$ on the model $G_k(n,r)$, it holds $2r < \sqrt{d}/2$ for all $d \geq 1$. In other words, $2r$ is smaller than the diameter of the $d$-dimensional torus. Therefore, we can consider two points $x_0,y_0 \in \mathbb{T}^d$ that are at a distance larger than $2r$. Let $U$ and $V$ be two open balls centered in $x_0$ and $y_0$ of small enough radius such that their distance is still larger than $2r$. Since $k \gg 1$, and vertices in $K$ are distributed uniformly at random, it follows that both $U$ and $V$ contain at least one vertex of the set $K$ with high probability. Therefore, there exists at least one pair of vertices in $K$ at a distance larger than $2r$, and the probability in \eqref{eq:bound_A} converges to 0.

In particular, if $\mu\ll n$ or $k\gg 1$ and $2^d\mu/n < 1-\varepsilon$, for some $\varepsilon>0$, then $\lim_{n\to\infty}\Prob{\mathcal{A}}=0$.
\end{proof}

Next, we bound the probabilities of events $\mathcal{B}_1$ and $\mathcal{B}_2$.

\begin{lemma}\label{prop:CNfailB1new}
In $G_k(n,r)$: \\
(i) If $\mu\ll n$ then
\begin{equation}\label{eq:CNfailB1new}
    \Prob{\mathcal{B}_1} \leq \mu n  e^{- c_{1,d} \phi_d r^d\,(n - k + 1 + k\,O(r^d))} +o(1), \nonumber
\end{equation}
for some $c_{1,d}>0$.

\noindent (ii) If $r= (\gamma/\phi_d)^{1/d}$ so that $\mu \sim \gamma n$ for $\gamma\in (0,1)$, there exists some constant $M_0$ such that for all $M\ge M_0$, 
\begin{equation}\label{eq:CNfailB1new-dense}
    \Prob{\mathcal{B}_1} \leq  M^d  e^{- M^{-d}n - k\log(1-M^{-d})}. \nonumber
\end{equation}
\end{lemma}

\begin{proof}
First, we prove (i). Define the events $A_{ij}:=\{N_i \cap N_j \text{ form a clique}\}$ for all $i<j$. Conditionally on the number of vertices, we bound the probability:
\begin{equation}
    \Prob{\mathcal{B}_1 \mid N_n = m} \leq \sum_{1\leq i<j\leq m} \Prob{A_{ij}, d(X_i,X_j) \leq r \mid N_n = m} \nonumber
\end{equation}
The common neighbours of $i$ and $j$ must have their positions inside the intersection of two balls of radius $r$ drawn around $X_i$ and $X_j$ (as shown in Figure \ref{fig: common-neighbours}). However, not all pairs of vertices in the intersection are necessarily connected to each other, as the diameter of the intersection is larger than $r$. For example, assume $i$ and $j$ are at distance $r$, and two vertices $v,w \not \in K$ have positions $X_{v}$ and $X_{w}$ inside the regions $R_1$ and $R_2$, respectively (see Figure \ref{fig: common-neighbours-area-algo-desc}). Then, $v,w \in N_i \cap N_j$, but $(v,w) \not \in E$ as their distance is larger than $r$. 
\begin{figure}[t]
    \centering
    \includegraphics[width = 0.6\textwidth]{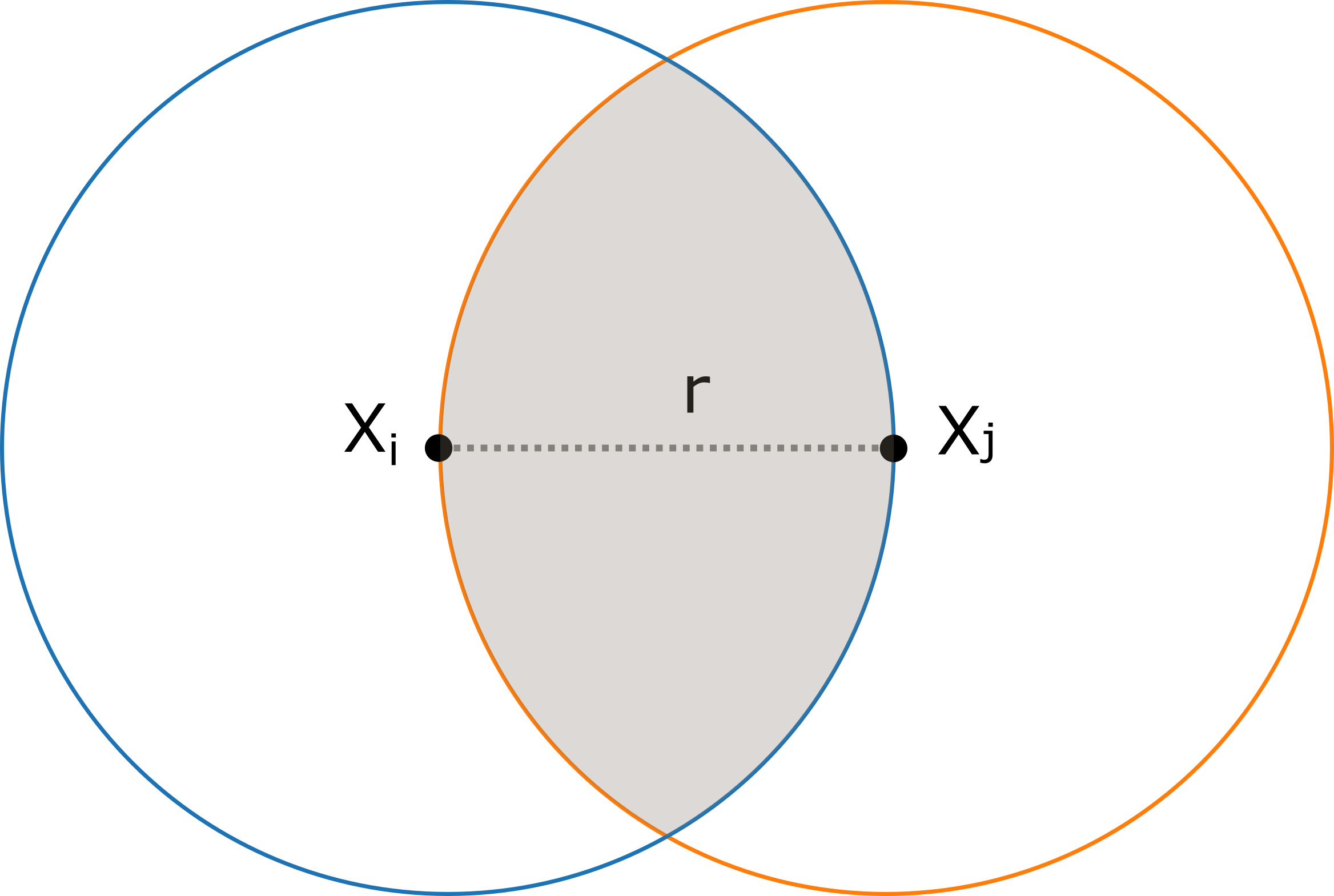}
    \caption{Representation for $d=2$ of the area that can contain common neighbours of $i$ and $j$}
    \label{fig: common-neighbours}
\end{figure}

At any fixed dimension $d$, the volume of the two regions $R_1$, $R_2$ in Figure \ref{fig: common-neighbours-area-algo-desc} has value $c_{1,d} \phi_d r^d$,
with $c_{1,d} > 0$, and such regions become larger as the vertices $i$ and $j$ are closer. Therefore, a necessary condition for the vertices in $N_i \cap N_j$ to form a clique is that either a region of size larger or equal to $R_1$ or $R_2$ contains no vertex of $V \setminus K$. Thus, 
\begin{align}
\nonumber
    &\Prob{A_{ij},d(X_i,X_j)\leq r \mid N_n = m} = \frac{\mu}{m}\Prob{A_{ij} \mid d(X_i,X_j)\leq r, N_n = m} \\
\nonumber
    & \qquad \leq \frac{\mu}{m}\Prob{\{\forall v \in V \setminus K, X_v \not \in R_1 \} \cup \{\forall v \in V \setminus K, X_v \not \in R_2\} \mid N_n =m} \\
\nonumber
    &\qquad \leq \frac{\mu}{m}2\Prob{\forall v \in V \setminus K, X_v \not \in R_1 \mid N_n = m} = \frac{2\mu}{m}(1-c_{1,d} \phi_d r^d)^{\max\{(m-k),0\}}.
\end{align}       
Thus, we obtain
$$\Prob{\mathcal{B}_1 \mid N_n = m} \leq \mu m (1-c_{1,d} \phi_d r^d)^{\max\{(m-k),0\}}$$
Finally, integrating over $N_n$, and recalling the Poisson Stein identity $\E{N_n f(N_n)}=n\E{f(N_n+1)}$, we get
\begin{align}
    \Prob{\mathcal{B}_1} &= \E{\Prob{\mathcal{B}_1 \mid N_n}} = \mu n \E{(1-c_{1,d} \phi_d r^d)^{\max\{(N_n-k + 1),0\}}} \nonumber\\
    &= \mu n e^{- c_{1,d} \phi_d r^d\,n-\log(1-c_{1,d} \phi_d r^d)\,(k - 1)} + o(1) \nonumber
\end{align}
where the last step follows from the explicit formula for the generating function of the Poisson random variable. Expanding $\log(1-c_{1,d} \phi_d r^d)$, this gives the desired result.
 
Next, we prove (ii). We do it in a standard way by dividing the space into $M^d$ cubes with side $1/M$ where $M\ge M_0$, and $M_0$ is such that the volume $c_{1,d} \phi_d r^d$ completely covers one of the $M_0^d$ cubes. Then, with the same argument as above, and using the union bound over $M^d$ cubes, we obtain that 
\begin{align*}
    \Prob{\mathcal{B}_1}&\le \PP(\exists\; \text{cube of size $1/M^d$ that contains no vertices of $V\backslash K$})\\
    & \le  M^d  e^{- M^{-d}n - k\log(1-M^{-d})}.
\end{align*}
This completes the proof. 
\end{proof}

It remains to bound the probability of $\mathcal{B}_2$. We do it in the following lemma.

\begin{lemma}\label{prop:CNfailB2new} In $G_k(n,r)$, if $\mu\ll n$ and $k \ll n/\mu$, then     
\begin{equation}\label{eq:CNfailB2new}
        \Prob{\mathcal{B}_2} \leq
        \begin{cases}
        \displaystyle \frac{\mu n}{2} \cdot \frac{(c_{2,d} \mu)^{k-2} e^{- c_{2,d} \mu}}{(k-2)!}+o(1) &\text{if $k-2 \leq c_{2,d} \mu$}, \\
        \displaystyle \frac{\mu n}{2} \cdot \frac{((k-2)/e)^{k-2}}{(k-2)!}+o(1) &\text{if $c_{2,d} \mu < k-2 < \mu$}, \\
        \displaystyle \frac{\mu n}{2} \cdot \frac{\mu^{k-2} e^{- \mu}}{(k-2)!}+o(1) &\text{if $k-2 \geq \mu$},
        \end{cases}
    \end{equation}
    for some $c_{2,d}>0$.
\end{lemma}

\begin{proof}
We bound the probability:
\begin{equation} \label{eq:upper_bound_CN_fail}
    \Prob{\mathcal{B}_2} \leq \sum_{(i,j) \in E \setminus K^2} \Prob{|Z_{ij}| = k-2}. \nonumber
\end{equation}
First, observe that for any edge $(i,j) \in E \setminus K^2$ the distance between the two endpoints is $x:= d(X_i, X_j) \leq r$. For ease of notation, we denote by $A(x,r) = \mathcal{V}(B(X_i, r) \cap B(X_j,r))$ the volume of the intersection of the two balls, which only depends on $x$ and $r$. Let us now distinguish two cases:
\begin{description}
    \item[Case 1.] When $i,j \in V \setminus K$, their common neighbours can only be found in the intersection $B(X_i, r) \cap B(X_j,r)$. In this case,
    \begin{align}\label{eq:poisson_ub}
        \Prob{|Z_{ij}| = k-2} = \Prob{\text{Poi}(n A(x,r)) = k-2}. \nonumber
    \end{align}  
    \item[Case 2.] When $i \not \in K$ and $j \in K$ (or vice versa using a symmetric argument), the common neighbours can be found in the intersection $B(X_i, r) \cap B(X_j,r)$ if they are not part of the planted clique, or they can be found in the entire ball $B(X_i,r)$ if they are part of the planted clique. Then 
    \begin{align}
        \Prob{|Z_{ij}| = k-2} = \Prob{\text{Poi}(n A(x,r)) + W = k-2}, \nonumber
    \end{align}
    where $W$ is the number of clique vertices in $B(X_i, r) \setminus B(X_j,r)$.
    \end{description}
In {\bf Case 1}, we will bound $\Prob{\mathcal{B}_2}$ by taking the supremum over $0 \leq x \leq r$:
\[\Prob{\text{Poi}(n A(x,r)) = k-2}\le \sup_{x\in(0,r)}\Prob{\text{Poi}(n A(x,r)) = k-2}. \]
Observe that, when $x$ increases from $0$ to $r$, the Poisson parameter $nA(x,r)$ decreases from $\mu$ to $c_{2,d}\mu$, with $c_{2,d} > 0$. Then, we have the following:
    \begin{enumerate}
        \item[(i)] If $k-2 \leq c_{2,d}\mu$, then 
        \begin{align*}
            \sup_{x\in(0,r)}\Prob{\text{Poi}(n A(x,r)) = k-2} &= \Prob{\text{Poi}(c_{2,d} \mu) = k-2} = \frac{e^{-c_{2,d}\mu} (c_{2,d}\mu)^{k-2}}{(k-2)!}.
        \end{align*}
        \item[(ii)] If $k-2 \geq \mu$, then 
        \begin{align*}
            \sup_{x\in(0,r)}\Prob{\text{Poi}(n A(x,r)) = k-2} & = \Prob{\text{Poi} (\mu) = k-2} = \frac{e^{-\mu} \mu^{k-2}}{(k-2)!}.
        \end{align*}
        \item[(iii)] If $c_{2,d} \mu < k-2 < \mu$, then 
        \begin{align*}
            \sup_{x\in(0,r)}\Prob{\text{Poi}(n A(x,r)) = k-2} & = \Prob{\text{Poi} (k-2) = k-2} = \frac{((k-2)/e)^{k-2}}{(k-2)!}.
        \end{align*}
    \end{enumerate}

\noindent In {\bf Case 2}, we need some extra steps.  
\begin{enumerate}
        \item[(i)] If $k-2 \leq c_{2,d}\mu$, we use the fact that the Poisson density is monotone on $[0,nA(x,r)]$:
        \begin{align*}
            &\Prob{\text{Poi}(nA(x,r))+W = k-2}\le \sup_{x\in(0,r)}\max_{0\le l\le k-2}\Prob{\text{Poi}(n A(x,r)) = l} \\
            &\quad =\max_{0\le l\le k-2}\Prob{\text{Poi}(c_{2,d}\mu) = l} = \Prob{\text{Poi}(c_{2,d} \mu) = k-2} = \frac{e^{-c_{2,d}\mu} (c_{2,d}\mu)^{k-2}}{(k-2)!}.
        \end{align*}
        \item[(ii)] If $k-2 \geq \mu$, then we use the Markov inequality and the fact that $\E{W}\le k\mu/n$:
        \begin{align*}
            &\Prob{\text{Poi}(nA(x,r))+W = k-2}\le \Prob{\text{Poi}(nA(x,r)) = k-2}+\Prob{W\ge 1}\\& 
            \quad \le  \Prob{\text{Poi}(\mu) = k-2} + \frac{k\mu}{n} = \frac{e^{-\mu} (\mu)^{k-2}}{(k-2)!} + \frac{k\mu}{n} = \frac{e^{-\mu} (\mu)^{k-2}}{(k-2)!} + o(1).
        \end{align*}
        
        \item[(iii)] If $c_{2,d} \mu < k-2 < \mu$, then 
        \begin{align*}
            \sup_{x\in(0,r)}\Prob{\text{Poi}(n A(x,r)) +W = k-2} &\le \Prob{\text{Poi} (k-2) = k-2}\\ &= \frac{(k-2)^{k-2}e^{-(k-2)}}{(k-2)!}.
        \end{align*}
    \end{enumerate} 
The result follows by conditioning on the number of vertices, applying the union bound, and finally integrating over the number of vertices, verbatim with Lemma~\ref{prop:CNfailB1new}.
\end{proof}

Before proving the main result for the CN-algorithm, we want to highlight that the bound in \eqref{eq:CNfailB2new} is not always informative. Indeed, when $c_{2,d} \mu < k-2 < \mu$, we can consider separately the two cases:
\begin{itemize}
    \item If $\mu = \omega(1)$, then Stirling's approximation yields 
    \begin{equation}
    \frac{\mu n}{2(k-2)!}\left(\frac{k-2}{e}\right)^{k-2}  = (1 + o(1)) \frac{\mu n}{\sqrt{2 \pi(k-2)}} \nonumber
    \end{equation}
    which diverges for $n \to \infty$.
    \item If $\mu = O(1)$, then $\frac{\mu n}{2(k-2)!}\left(\frac{k-2}{e}\right)^{k-2} = n O(1)$, which again diverges.
\end{itemize}
The conclusion here is that event $\mathcal{B}_2$ may not vanish if $k$ is in the interval $(c_{2,d} \mu, \mu)$. Recall that if additionally $\mu$ is too small (that is, $\mu = O(\log n + \log\log n)$) the quantity $\Prob{\mathcal{B_1}}$ may not be vanishing, and Lemma \ref{prop:CNfailB1new} becomes not informative. In turn, we cannot guarantee the CN-algorithm's success in this region, as the sufficient condition $\min(\Prob{\mathcal{B_1},\mathcal{B_2}}) \to 0$ is not proved. The intuition that such a condition is not verified is further confirmed by the numerical experiments in Section \ref{sec:numerical_experiments}. Indeed, Figure \ref{fig:compare_mu5} and \ref{fig:compare_mu20} show a sudden drop in the CN-algorithm performance, exactly when the parameters lie inside the orange band in Figure \ref{fig:success_CN_algo} that delineates the region $k \in (c_{2,d} \mu, \mu)$.

Combining the previous lemmas, we can now demonstrate the main result for the CN algorithm.

\begin{proof}[Proof of Theorem \ref{thm:successCNalgo}]
By \eqref{eq:failure-CN-AB12}, we need that $\Prob{\mathcal{A}}=o(1)$ and either $\Prob{\mathcal{B}_1}=o(1)$  or $\Prob{\mathcal{B}_2}=o(1)$. In particular,
\begin{itemize}
    \item Lemma~\ref{prop:CNfailA} implies that $\Prob{\mathcal{A}}=o(1)$ in not ill-imposed parameter range.
    \item From Lemma~\ref{prop:CNfailB1new}, we conclude that Equation \eqref{eq:cond2successCN} is the necessary condition for $\Prob{\mathcal{B}_1}=o(1)$. Moreover, notice that the upper bound in Lemma~\ref{prop:CNfailB1new}(ii) goes to zero for large enough $M$ because $k<\alpha n$, and that $\mu\sim \gamma n$ automatically satisfies \eqref{eq:cond2successCN}.  
    \item Lastly, from Lemma~\ref{prop:CNfailB2new} we see that condition \eqref{eq:cond3successCN} implies $\Prob{\mathcal{B}_2}=o(1)$.
\end{itemize}
\end{proof}

\section{Conclusions and future research}\label{sec:conclusion}

In the present paper, we have studied a variation of the classical planted clique problem that was first formulated for Erd\H{o}s--R\'enyi graphs. We showed that the same task for geometric graphs provides a more diverse setting. Moreover, it turned out that the geometrical properties of a graph create the possibility of recovering very small planted cliques (even planted edges), which is rather remarkable. 

Two algorithms were studied in the paper from both theoretical and practical perspectives. The first one, the VD-algorithm, is rather simple and was introduced in \cite{kuvcera1995expected}. We have conducted a comprehensive theoretical analysis of this algorithm for random geometric graphs, revealing asymptotic thresholds for exact recovery. The second one, the CN-algorithm, was shown to recover the planted clique for a wider range of parameters of random geometric graphs. Some numerical experiments were conducted on random geometric graphs. They confirm the obtained theoretical results and demonstrate that the devised algorithms can be applied in practice to real-life tasks. 

Several open questions can be further investigated. Firstly, the planted clique problem can be easily reformulated for the case of soft random geometric graphs, such as Waxman graphs. Another direction would be to consider RGG with no fixed dimension. Indeed, when $d$ grows to infinity, geometric random graphs converge to Erd\H{o}s--R\'enyi graphs \cite{bubeck2016testing,devroye2011high}; hence, it would be interesting to understand the impact of high-dimensionality in the algorithmic performance.\\
Another possible direction of research is to consider that a planted clique can be defined in a slightly different manner. For instance, one natural way to build a planted clique is geometrical: take a circle in the unit square and connect all the vertices with positions inside it. It is clear that the CN-algorithm, for example, is useless in this case. We do not know if there is another algorithm that solves this problem completely, but all ideas in this direction will be of great interest. \\
Ultimately, one can focus on refining the existing algorithms in terms of complexity and practical implementation.

\section*{Declarations}

\paragraph{Acknowledgements}
We thank Joost Jorritsma for insightful discussions of the proof of Lemma \ref{prop:CNfailB2new}.

\paragraph{Funding}
RM was supported by PNRR MUR project GAMING ``Graph Algorithms and MinINg for Green agents'' (PE0000013, CUP D13C24000430001).  The work of NL is supported by the
Netherlands Organisation for Scientific Research (NWO) through the Gravitation NETWORKS grant 024.002.003, and by the National Science Foundation under Grant No. DMS-1928930 while the author was in residence at the Simons Laufer Mathematical Sciences Institute in Berkeley, California, during the Spring semester 2025.

\paragraph{Conflict of interest/Competing interests}
The authors declare no potential conflict of interest.

\paragraph{Code availability}
The produced code is available on request.

\bibliographystyle{unsrt}  
\bibliography{references}  

\end{document}